\documentclass[11pt, reqno]{amsart}
\usepackage[T1]{fontenc}
\usepackage{amsfonts}
\usepackage{amsmath}
\usepackage{amsthm}
\usepackage{amssymb}
\usepackage{geometry}
\usepackage{graphicx}
\usepackage{xcolor}
\usepackage{mathtools}
\usepackage[colorlinks=true, linkcolor=blue]{hyperref}
\usepackage{cleveref}
\usepackage[english]{babel}
\usepackage{lineno}

\newtheorem{theorem}{Theorem}[section]

\newtheorem{proposition}[theorem]{Proposition}

\usepackage{tikz}
\usetikzlibrary{positioning}
\usetikzlibrary{decorations,arrows}
\usetikzlibrary{decorations.markings}
\numberwithin{equation}{section}

\newcommand{\N}{\mathbb N}
\newcommand{\Z}{\mathbb Z}
\newcommand{\R}{\mathbb R}

\usepackage{rustic}
\usepackage[T1]{fontenc}

\emergencystretch=\maxdimen
\title{A note on the threshold numbers of cycles}
\author{Runze Wang}
\address[]{Department of Mathematical Sciences, University of Memphis, Memphis, TN 38152, USA}
\email{runze.w@hotmail.com}
\thanks{}
\date{\today}
\subjclass[]{}

\begin{document}

\sloppy

\begin{abstract}
    A graph $G=(V,E)$ is said to be a \textit{$k$-threshold graph} with \textit{thresholds} $\theta_1<\theta_2<...<\theta_k$ if there is a map $r: V \longrightarrow \mathbb{R}$ such that $uv\in E$ if and only if $\theta_i\le r(u)+r(v)$ holds for an odd number of $i\in [k]$. The \textit{threshold number} of $G$, denoted by $\Theta(G)$, is the smallest positive integer $k$ such that $G$ is a $k$-threshold graph. In this paper, we determine the exact threshold numbers of cycles by proving
    \[ \Theta(C_n)=\begin{cases} 
          1 & if\ n=3, \\
          2 & if\ n=4, \\
          4 & if\ n\ge 5,
       \end{cases}
    \]
    where $C_n$ is the cycle with $n$ vertices.
\end{abstract}

\maketitle

\section{Introduction}
In this paper, we only consider finite simple graphs. And for $a,b\in \Z$ with $a<b$, we denote $\{i\in \Z: a\le i\le b\}$ by $[a,b]$; for $c\in \N$, we denote $\{j\in \N: 1\le j\le c\}$ by $[c]$.

A graph $G=(V,E)$ is said to be a \textit{threshold graph} if there is a map $f: V \longrightarrow \R$ such that $uv\in E$ if and only if $f(u)+f(v)\ge 0$. Threshold graphs were introduced by Chv\'atal and Hammer \cite{CH77} in 1977, and have been extensively studied \cite{Go,GJ,GT,MP} since then.

Jamison and Sprague \cite{JS} introduced multithreshold graphs as a generalization of threshold graphs. A graph $G=(V,E)$ is said to be a \textit{$k$-threshold graph} with \textit{thresholds} $\theta_1<\theta_2<...<\theta_k$ if there is a map $r: V \longrightarrow \R$ such that $uv\in E$ if and only if $\theta_i\le r(u)+r(v)$ holds for an odd number of $i\in [k]$. We say $r(v)$ is the \textit{rank} of $v$. We say such a rank assignment $r: V \longrightarrow \R$ is a \textit{$(\theta_1,\theta_2,...,\theta_k)$-representation} of $G$.

If $G$ is a $k$-threshold graph, then the $k$ thresholds divide $\R$ into $k+1$ parts in the form
\centerline{NO | YES | NO | YES | ...,}
and every edge rank sum lands in a YES part, every nonedge rank sum lands in a NO part.

Jamison and Sprague \cite{JS} proved that if a graph has $n$ vertices, then it is a $k$-threshold graph for some $k\le{n\choose 2}$. This means for any finite graph $G$, there is a smallest positive integer $\Theta(G)$, called the \textit{threshold number} of $G$, such that $G$ is a $\Theta(G)$-threshold graph.

For specific graphs: Jamison and Sprague \cite{JS} proved that the threshold number of any path is at most $2$, and the threshold number of any caterpillar (obtained by attaching leaves to vertices in a path) is also at most $2$; Chen and Hao \cite{CH22} determined the exact threshold numbers of $m$-partite graphs with every part having at least $m+1$ vertices; Kittipassorn and Sumalroj \cite{KS} determined the exact threshold numbers of multipartite graphs with every part having $3$ vertices or every part having $4$ vertices; Wang \cite{Wa} determined the exact threshold numbers of multipartite graphs with every part having at most $3$ vertices.

In this paper, we determine the exact threshold numbers of cycles.

\begin{theorem} \label{cycles}
    Let $n\ge 3$ be an integer, let $C_n$ be the cycle with $n$ vertices, then
    \[ \Theta(C_n)=\begin{cases} 
          1 & if\ n=3, \\
          2 & if\ n=4, \\
          4 & if\ n\ge 5.
       \end{cases}
    \]
\end{theorem}

\section{Proof of Theorem \ref{cycles}}

For $C_n$, the cycle with $n$ vertices, we randomly pick a vertex and label it $v_1$, then counterclockwise label the other vertices $v_2,v_3,...,v_n$. Note that we can continue this labeling process, so each vertex $v_i$ can also be represented as $v_{i+kn}$ for any $k\in \N$, e.g. $v_1$ and $v_{n+1}$ are the same vertex.

Let us prove Theorem \ref{cycles}.

\begin{proof}[Proof of Theorem \ref{cycles}]
    It is easy to check that $\Theta(C_3)=1$.

    For $C_n$ with $n\ge 4$, we first show that $\Theta(C_n)\ge 2$.
    
    In $C_n$ with $n\ge 4$, $v_1 v_2$ and $v_3 v_4$ are edges, $v_1 v_3$ and $v_2 v_4$ are nonedges. Assume by contradiction that $\Theta(C_n)=1$ and $r: V\longrightarrow \R$ is a $(\theta)$-representation of $C_n$ for some $\theta$. Then by the definition of multithreshold graphs, we have $r(v_1)+r(v_2)\ge \theta$, $r(v_3)+r(v_4)\ge \theta$, $r(v_1)+r(v_3)< \theta$, and $r(v_2)+r(v_4)< \theta$. But then we will have
    \begin{align*}
        (r(v_1)+r(v_2))+(r(v_3)+r(v_4))\ge 2\theta> (r(v_1)+r(v_3))+(r(v_2)+r(v_4)),
    \end{align*}
    contradiction. So $\Theta(C_n)\ge 2$.

    For $C_4$, if we let $\theta_1=0.9$ and $\theta_2=1.1$, and let $r(v_1)=0$, $r(v_2)=1$, $r(v_3)=0$, and $r(v_4)=1$, then every edge rank sum is $1$, and every nonedge rank sum is either $0$ or $2$, so we have a $(0.9,1.1)$-representation of $C_4$. So $\Theta(C_4)\le 2$, and hence $\Theta(C_4)=2$.
    
    Now let us assume $n\ge 5$.

    Firstly, we show $\Theta(C_n)\le 4$ by construction. Let $r(v_i)=(-1)^{i-1}i$ for $i\in [n-1]$, and let $r(v_n)=(-1)^{n-1}(n-0.5)$.

    \textbf{Case I.} $n$ is odd.
    
    In this case $r(v_n)=n-0.5$. We let $\theta_1=-1.1$, $\theta_2=1.1$, $\theta_3=r(v_1)+r(v_n)-0.1=n+0.4$, and $\theta_4=r(v_1)+r(v_n)+0.1=n+0.6$. We have $\theta_1<\theta_2<\theta_3<\theta_4$ because $n\ge 5$. Denote $\{\theta_1,\theta_2,\theta_3,\theta_4\}$ by $\mathcal{S}$.
    
    Then for an edge not involving $v_n$, its rank sum is $-1$ or $1$, which is between $\theta_1$ and $\theta_2$. For edge $v_{n-1}v_n$, its rank sum is $0.5$, which is between $\theta_1$ and $\theta_2$. For edge $v_1 v_n$, its rank sum is $1+(n-0.5)=n+0.5$, which is between $\theta_3$ and $\theta_4$. We have checked every edge rank sum is greater than or equal to an odd number of elements in $\mathcal{S}$. 
    
    For a nonedge not involving $v_n$, assume its two vertices are $v_i$ and $v_j$ with $i,j\in [n-1]$ and $|i-j|\neq 1$, then:
    \begin{itemize}
        \item If $r(v_i)+r(v_j)<0$, then by $|i-j|\neq 1$, we know that $r(v_i)+r(v_j)\le -3<\theta_1$, so $|\{s\in \mathcal{S}: s\le r(v_i)+r(v_j)\}|=0$, an even number.
        \item If $r(v_i)+r(v_j)>0$, then by $|i-j|\neq 1$, we know that $r(v_i)+r(v_j)\ge 3>\theta_2$. And $r(v_i)+r(v_j)$ is an integer, so it is either smaller than $\theta_3$ or greater than $\theta_4$. So $|\{s\in \mathcal{S}: s\le r(v_i)+r(v_j)\}|$ is $2$ or $4$, an even number.
    \end{itemize}

    And for a nonedge involving $v_n$, assume the other vertex of this nonedge is $v_i$ with $i\in [2, n-2]$, then:
    \begin{itemize}
        \item If $r(v_i)>0$, then $r(v_i)$ is at least $3$, so $r(v_i)+r(v_n)\ge 3+(n-0.5)=n+2.5>\theta_4$, so $|\{s\in \mathcal{S}: s\le r(v_i)+r(v_j)\}|=4$, an even number.
        \item If $r(v_i)<0$, then $-(n-3)\le r(v_i)\le -2$. Note that we have $-(n-3)\le -2$ because $n\ge 5$. And then we have $\theta_2<2.5\le r(v_i)+r(v_n)\le n-2.5<\theta_3$, so $|\{s\in \mathcal{S}: s\le r(v_i)+r(v_j)\}|=2$, an even number.
    \end{itemize}
    
    We have also checked every nonedge rank sum is greater than or equal to an even number of elements in $\mathcal{S}$. So we have a $(\theta_1,\theta_2,\theta_3,\theta_4)$-representation of $C_n$, and hence $\Theta(C_n)\le 4$.

    \textbf{Case II.} $n$ is even.
    
    In this case, $r(v_n)=-n+0.5$. We let $\theta_1=r(v_1)+r(v_n)-0.1=-n+1.4$, $\theta_2=r(v_1)+r(v_n)+0.1=-n+1.6$, $\theta_3=-1.1$, and $\theta_4=1.1$. We have $\theta_1<\theta_2<\theta_3<\theta_4$ because $n\ge 5$. Denote $\{\theta_1,\theta_2,\theta_3,\theta_4\}$ by $\mathcal{S}'$. The same as in Case I, we can check every edge rank sum is greater than or equal to an odd number of elements in $\mathcal{S}'$, and every nonedge rank sum is greater than or equal to an even number of elements in $\mathcal{S}'$. So we also have $\Theta(C_n)\le 4$ in this case.

    Now we have proved the upper bound $\Theta(C_n)\le 4$.
    
    And we already know $\Theta(C_n)\ge 2$. So by eliminating the possibility of $\Theta(C_n)=2$ or $\Theta(C_n)=3$, we can conclude that $\Theta(C_n)=4$.

    Firstly assume $\Theta(C_n)=2$. So we have two thresholds $\theta_1, \theta_2$, and a $(\theta_1, \theta_2)$-representation $r: V \longrightarrow \R$. Two thresholds divide $\R$ into three parts in the form
    
    \centerline{NO | YES | NO.}
    
    The two NO parts will be called smaller NO and larger NO. 
    
    We have three cases. The proofs of the first two follow the same pattern, the proof of the third one is slightly different.

    \textbf{Case i.} $2 \nmid n$.

    For vertices $v_1,v_2,v_3,v_4$: We know $v_1 v_2$ and $v_3 v_4$ are edges, so both $r(v_1)+r(v_2)$ and $r(v_3)+r(v_4)$ are in YES. We know $v_1 v_3$ and $v_2 v_4$ are nonedges, so both $r(v_1)+r(v_3)$ and $r(v_2)+r(v_4)$ are in NO. And by the fact that $(r(v_1)+r(v_2))+(r(v_3)+r(v_4))=(r(v_1)+r(v_3))+(r(v_2)+r(v_4))$, we know that $r(v_1)+r(v_3)$ and $r(v_2)+r(v_4)$ are in different NO's. Without loss of generality, we assume $r(v_1)+r(v_3)$ is in smaller NO and $r(v_2)+r(v_4)$ is in larger NO.

    Then for vertices $v_2,v_3,v_4,v_5$: Similarly we know that $r(v_2)+r(v_4)$ and $r(v_3)+r(v_5)$ are in different NO's, and we already assumed $r(v_2)+r(v_4)$ is in larger NO, so $r(v_3)+r(v_5)$ is in smaller NO.

    Repeatedly, we have that if $i$ is odd then $r(v_i)+r(v_{i+2})$ is in smaller NO, if $i$ is even then $r(v_i)+r(v_{i+2})$ is in larger NO. Now we take $i=n+1$, we know $n+1$ is even because $2 \nmid n$, so we have $r(v_{n+1})+r(v_{n+3})$ is in larger NO, but $v_{n+1}$ and $v_1$ represent the same vertex, $v_{n+3}$ and $v_3$ represent the same vertex, so $r(v_1)+r(v_3)$ is in larger NO, which contradicts our assumption that $r(v_1)+r(v_3)$ is in smaller NO.

    \textbf{Case ii.} $2 \mid n$, but $4 \nmid n$.

    Because $2 \mid n$, we know $\frac{n}{2}$ is an integer.

    For vertices $v_1,v_2,v_{\frac{n}{2}+1},v_{\frac{n}{2}+2}$: We know $v_1 v_2$ and $v_{\frac{n}{2}+1} v_{\frac{n}{2}+2}$ are edges, so both $r(v_1)+r(v_2)$ and $r(v_{\frac{n}{2}+1})+r(v_{\frac{n}{2}+2})$ are in YES. We know $v_1 v_{\frac{n}{2}+1}$ and $v_2 v_{\frac{n}{2}+2}$ are nonedges, so both $r(v_1)+r(v_{\frac{n}{2}+1})$ and $r(v_2)+r(v_{\frac{n}{2}+2})$ are in NO. And by the fact that $(r(v_1)+r(v_2))+(r(v_{\frac{n}{2}+1})+r(v_{\frac{n}{2}+2}))=(r(v_1)+r(v_{\frac{n}{2}+1}))+(r(v_2)+r(v_{\frac{n}{2}+2}))$, we know that $r(v_1)+r(v_{\frac{n}{2}+1})$ and $r(v_2)+r(v_{\frac{n}{2}+2})$ are in different NO's. Without loss of generality, we assume $r(v_1)+r(v_{\frac{n}{2}+1})$ is in smaller NO and $r(v_2)+r(v_{\frac{n}{2}+2})$ is in larger NO.

    Then for vertices $v_2,v_3,v_{\frac{n}{2}+2},v_{\frac{n}{2}+3}$: Similarly we know that $r(v_2)+r(v_{\frac{n}{2}+2})$ and $r(v_3)+r(v_{\frac{n}{2}+3})$ are in different NO's, and we already assumed $r(v_2)+r(v_{\frac{n}{2}+2})$ is in larger NO, so $r(v_3)+r(v_{\frac{n}{2}+3})$ is in smaller NO.

    Repeatedly, we have that if $i$ is odd then $r(v_i)+r(v_{\frac{n}{2}+i})$ is in smaller NO, if $i$ is even then $r(v_i)+r(v_{\frac{n}{2}+i})$ is in larger NO. Now we take $i=\frac{n}{2}+1$, we know $\frac{n}{2}+1$ is even because $4 \nmid n$, so we have $r(v_{\frac{n}{2}+1})+r(v_{n+1})$ is in larger NO, but $v_{n+1}$ is the same vertex as $v_1$, so $r(v_{\frac{n}{2}+1})+r(v_1)$ is in larger NO, which contradicts our assumption that $r(v_1)+r(v_{\frac{n}{2}+1})$ is in smaller NO.

    \textbf{Case iii.} $4 \mid n$.
    
    For vertices $v_1,v_2,v_\frac{n}{2},v_{\frac{n}{2}+1}$: We know $v_1 v_2$ and $v_\frac{n}{2} v_{\frac{n}{2}+1}$ are edges, so both $r(v_1)+r(v_2)$ and $r(v_\frac{n}{2})+r(v_{\frac{n}{2}+1})$ are in YES. We know $v_1 v_\frac{n}{2}$ and $v_2 v_{\frac{n}{2}+1}$ are nonedges, so both $r(v_1)+r(v_\frac{n}{2})$ and $r(v_2)+r(v_{\frac{n}{2}+1})$ are in NO. And by the fact that $(r(v_1)+r(v_2))+(r(v_\frac{n}{2})+r(v_{\frac{n}{2}+1}))=(r(v_1)+r(v_\frac{n}{2}))+(r(v_2)+r(v_{\frac{n}{2}+1}))$, we know that $r(v_1)+r(v_\frac{n}{2})$ and $r(v_2)+r(v_{\frac{n}{2}+1})$ are in different NO's. Without loss of generality, we assume $r(v_1)+r(v_\frac{n}{2})$ is in smaller NO and $r(v_2)+r(v_{\frac{n}{2}+1})$ is in larger NO.

    Then for vertices $v_2,v_3,v_{\frac{n}{2}+1},v_{\frac{n}{2}+2}$: Similarly we know that $r(v_2)+r(v_{\frac{n}{2}+1})$ and $r(v_3)+r(v_{\frac{n}{2}+2})$ are in different NO's, and we already assumed $r(v_2)+r(v_{\frac{n}{2}+1})$ is in larger NO, so $r(v_3)+r(v_{\frac{n}{2}+2})$ is in smaller NO.

    Repeatedly, we have that if $i$ is odd then $r(v_i)+r(v_{\frac{n}{2}+i-1})$ is in smaller NO, if $i$ is even then $r(v_i)+r(v_{\frac{n}{2}+i-1})$ is in larger NO. Now we take $i=\frac{n}{2}+1$, we know $\frac{n}{2}+1$ is odd because $4 \mid n$, so $r(v_{\frac{n}{2}+1})+r(v_n)$ is in smaller NO. But then for vertices $v_\frac{n}{2},v_{\frac{n}{2}+1},v_n,v_1$: We know $v_\frac{n}{2} v_{\frac{n}{2}+1}$ and $v_n v_1$ are edges, so both $r(v_\frac{n}{2})+r(v_{\frac{n}{2}+1})$ and $r(v_n)+r(v_1)$ are in YES. We know $v_1 v_\frac{n}{2}$ and $v_{\frac{n}{2}+1} v_n$ are nonedges, and both $r(v_1)+r(v_\frac{n}{2})$ and $r(v_{\frac{n}{2}+1})+r(v_n)$ are in smaller NO. But then we will have $(r(v_\frac{n}{2})+r(v_{\frac{n}{2}+1}))+(r(v_n)+r(v_1))>(r(v_1)+r(v_\frac{n}{2}))+(r(v_{\frac{n}{2}+1})+r(v_n))$, contradiction.

    In each of the three cases, we get a contradiction, so $\Theta(C_n)$ cannot be $2$.

    Then assume $\Theta(C_n)=3$. So we have three thresholds $\theta_1, \theta_2, \theta_3$, and a $(\theta_1, \theta_2, \theta_3)$-representation $r: V \longrightarrow \R$. Three thresholds divide $\R$ into four parts in the form
    
    \centerline{NO | YES | NO | YES.}
    
    The two NO parts will be called smaller NO and larger NO, the two YES parts will be called smaller YES and larger YES.

    Firstly we note that there must be an edge rank sum in larger YES, as otherwise two thresholds would be enough, $\Theta(C_n)=3$ would not hold. We may assume $r(v_1)+r(v_2)$ is in larger YES.

    We also make the following claim.
    
    \textbf{Claim.} Assume $v_i v_{i+1}$ and $v_j v_{j+1}$ are disjoint edges in $C_n$, then $r(v_i)+r(v_{i+1})$ and $r(v_j)+r(v_{j+1})$ cannot both be in larger YES.

    \begin{proof}[Proof of Claim.]
        Because $v_i v_{i+1}$ and $v_j v_{j+1}$ are disjoint edges, we can see that both $v_i v_j$ and $v_{i+1} v_{j+1}$ are nonedges. Assume by contradiction that both $r(v_i)+r(v_{i+1})$ and $r(v_j)+r(v_{j+1})$ are in larger YES. Then by the equation $(r(v_i)+r(v_{i+1}))+(r(v_j)+r(v_{j+1}))=(r(v_i)+r(v_j))+(r(v_{i+1})+r(v_{j+1}))$, we know that $\max\{r(v_i)+r(v_j), r(v_{i+1})+r(v_{j+1})\}\ge \min\{r(v_i)+r(v_{i+1}), r(v_j)+r(v_{j+1})\}$, so at least one of $r(v_i)+r(v_j)$ and $r(v_{i+1})+r(v_{j+1})$ is in larger YES. However, both $v_i v_j$ and $v_{i+1} v_{j+1}$ are nonedges, contradiction.
    \end{proof}
    
    By this claim, we know that edge $v_2 v_3$ and edge $v_n v_1$ cannot both have rank sums in larger YES, because they are disjoint edges. Moreover, any edge $v_i v_{i+1}$ with $i\in [3, n-1]$ cannot have rank sum in larger YES, because $v_1 v_2$ and $v_i v_{i+1}$ are disjoint edges.

    We have two cases:

    \textbf{Case 1.} None of $v_2 v_3$ and $v_n v_1$ has rank sum in larger YES.

    We have that both $r(v_2)+r(v_3)$ and $r(v_n)+(v_1)$ are in smaller YES, so by a similar argument as in the proof of $\Theta(C_n)\neq 2$, we know $r(v_1)+r(v_3)$ and $r(v_2)+r(v_n)$ are in different NO's. Without loss of generality, we may assume $r(v_1)+r(v_3)$ is in smaller NO. 
    
    Because $v_2 v_4$ is a nonedge and $r(v_1)+r(v_2)$ is in larger YES, we have $r(v_2)+r(v_4)<r(v_1)+r(v_2)$, so $r(v_4)<r(v_1)$, which implies $r(v_3)+r(v_4)<r(v_3)+r(v_1)$, and we assumed $r(v_3)+r(v_1)$ is in smaller NO, so $r(v_3)+r(v_4)$ is also in smaller NO, which contradicts the fact that $v_3 v_4$ is an edge.

    \textbf{Case 2.} Exactly one of $v_2 v_3$ and $v_n v_1$ has rank sum in larger YES.

    By symmetry, we may assume $r(v_2)+r(v_3)$ is in larger YES. Now, $r(v_1)+r(v_2)$ and $r(v_2)+r(v_3)$ are in larger YES, and by the claim we just made, all other edge rank sums are in smaller YES. In particular, both $r(v_3)+r(v_4)$ and $r(v_n)+(v_1)$ are in smaller YES, so by a similar argument as in the proof of $\Theta(C_n)\neq 2$, we know $r(v_3)+r(v_n)$ and $r(v_4)+r(v_1)$ are in different NO's. Without loss of generality, we may assume $r(v_3)+r(v_n)$ is in smaller NO.

    Now we look at vertices $v_2,v_3,v_{n-1},v_n$: We know $v_2 v_3$ is an edge with $r(v_2)+r(v_3)$ in larger YES, $v_{n-1} v_n$ is an edge with $r(v_{n-1})+r(v_n)$ in smaller YES. Now, because $r(v_3)+r(v_n)$ is in smaller NO, we have $r(v_3)+r(v_n)<r(v_{n-1})+r(v_n)$. Because $v_2 v_{n-1}$ is a nonedge, we have $r(v_2)+r(v_{n-1})$ is in NO (larger or smaller), and we know $r(v_2)+r(v_3)$ is in larger YES, so $r(v_2)+r(v_{n-1})<r(v_2)+r(v_3)$. But then we have $(r(v_3)+r(v_n))+(r(v_2)+r(v_{n-1}))<(r(v_{n-1})+r(v_n))+(r(v_2)+r(v_3))$, contradiction.

    Note that if $n=5$, then $v_4$ and $v_{n-1}$ are the same vertex, but this does not affect our argument.

    In either case, we get a contradiction, so $\Theta(C_n)$ cannot be $3$.

    In summary, for $n\ge 5$, We have showed $\Theta(C_n)\ge 2$, $\Theta(C_n)\le 4$, $\Theta(C_n)\neq 2$, and $\Theta(C_n)\neq 3$, thus $\Theta(C_n)=4$.
\end{proof}

\section{Remarks}
Actually there is another way to see that $\Theta(C_n)\le 4$, which is given by combining the following two propositions and the fact that we can get $C_n$ by adding an edge to $P_n$.

\begin{proposition}[Jamison and Sprague \cite{JS}]
    Let $n$ be a positive integer, let $P_n$ be the path with $n$ vertices, then $\Theta(P_n)\le 2$.
\end{proposition}

\begin{proposition}[Jamison and Sprague \cite{JS}]
    Let $G=(V,E)$ be a graph with $\Theta(G)=k$, and let $u,v\in V$ be distinct vertices.

    (a) If $uv$ is an edge, then the graph $G-uv$ obtained by deleting edge $uv$ from $E$ has threshold number at most $k+2$.

    (b) If $uv$ is a nonedge, then the graph $G+uv$ obtained by adding edge $uv$ to $E$ has threshold number at most $k+2$.
\end{proposition}

The author adopted the constructive proof, as it helps readers understand how things work.

\end{document}